\documentclass[12pt]{article}
\usepackage{amsmath}
\usepackage{amssymb}
\usepackage{latexsym}
\usepackage{amsthm}
\usepackage{mathrsfs}
\usepackage{enumitem}
\usepackage[colorlinks,
            linkcolor=red,
            anchorcolor=blue,
            citecolor=green]{hyperref}
\usepackage{mma}

\newtheorem{thm}{Theorem}[section]
\newtheorem{lem}[thm]{Lemma}
\newtheorem{cor}[thm]{Corollary}
\newtheorem{conj}[thm]{Conjecture}

\newcommand{\sep}{\preceq}

\newcommand{\R}{\mathbb{R}}
\def\newop#1{\expandafter\def\csname #1\endcsname{\mathop{\rm
#1}\nolimits}}
\newop{sgn}

\begin{document}

\begin{center}
{\large \bf  The Real-rootedness of Generalized Narayana Polynomials}
\end{center}

\begin{center}
Herman Z.Q. Chen$^{1}$, Arthur L.B. Yang$^{2}$, Philip B. Zhang$^{3}$\\[6pt]
\par\end{center}

\begin{center}
$^{1,2}$Center for Combinatorics, LPMC\\
 Nankai University, Tianjin 300071, P. R. China

$^{3}$College of Mathematical Science \\
Tianjin Normal University, Tianjin  300387, P. R. China\\[6pt]
\par\end{center}

\begin{center}
Email: $^{1}$\texttt{zqchern@163.com, $^{2}$\texttt{yang@nankai.edu.cn}},
$^{3}$\texttt{zhangbiaonk@163.com}
\par
\end{center}

\noindent \emph{Abstract.}
In this paper, we prove the real-rootedness of two classes of generalized Narayana polynomials:
one arising as the $h$-polynomials of the generalized associahedron associated to the finite Weyl groups, the other arising in the study of the infinite log-concavity of the Boros-Moll polynomials. For the former, Br\"{a}nd\'{e}n has already proved that these $h$-polynomials have only real zeros. We establish certain recurrence relations for the two classes of Narayana polynomials, from which we derive the real-rootedness. To prove the real-rootedness, we use a sufficient condition, due to Liu and Wang, to determine whether two polynomials have interlaced zeros. The recurrence relations are verified with the help of the Mathematica package \textit{HolonomicFunctions}.

\noindent \emph{AMS Classification 2010:} 05A15, 26C10


\section{Introduction}

For any nonnegative integers $n$ and $m$, let
\begin{align}
N_{A_{n}}(x) & =  \frac {1} {n+1}\sum_{k=0}^{n}\binom {n+1} {k} \binom {n+1} {k+1}x^k,\label{eq-nara1}\\[6pt]
N_{B_{n}}(x)  & =  \sum_{k=0}^n \binom n k \binom n kx^k,\label{eq-nara2}\\[6pt]
N_{D_{n}}(x)  & =  N_{B_{n}}(x) - n x N_{A_{n-2}}(x), \quad \mbox{for $n\geq 2$},\label{eq-nara3}\\[6pt]
N_{n,m}(x) & =  \sum_{k=0}^{n}\left(\binom{n}{k}\binom{m}{k}-\binom{n}{k+1}\binom{m}{k-1}\right)x^k.\label{eq-nara4}
\end{align}
The polynomial $N_{A_{n}}(x)$ is the classical Narayana polynomial. It is well
known that both $N_{A_{n}}(x)$ and $N_{B_{n}}(x)$ have only real zeros. The
real-rootedness of $N_{D_{n}}(x)$ was proved by Br\"{a}nd\'{e}n
\cite{Braenden2006linear}. In this paper, we shall give a new proof of the
real-rootedness of $N_{D_{n}}(x)$.
It is easy to verify that both $N_{n,n}(x)$ and $N_{n+1,n}(x)$ are just the
polynomial $N_{A_{n}}(x)$. While,
it seems that the polynomials $N_{n,m}(x)$ were not well studied for general $n$
and $m$. In this paper, we shall prove that the polynomials $N_{n,m}(x)$ have
only real zeros for any $n$ and $m$. Let us first review some backgrounds of the
polynomials $N_{D_{n}}(x)$ and $N_{n,m}(x)$.

The polynomials $N_{D_{n}}(x)$ arose in the study of the generalized associahedron
associated to finite Weyl groups, see Fomin and Zelevinsky \cite{Fomin2003Y}.
Let $W$ be a finite Weyl group, and denote by $N_W(x)$ the $h$-polynomial of dual complex of the generalized associahedra associated to $W$.
For the classical Weyl groups of type $A_n$, $B_n$ and $D_n$, these polynomials are given by
$N_{A_{n}}(x)$, $N_{B_{n}}(x)$ and $N_{D_{n}}(x)$, respectively. For more information, we refer the reader to \cite{Armstrong2009Generalized, Athanasiadis2005refinement, FominRoot, Fomin2005Generalized}.
The polynomials $N_{W}(x)$ are called generalized Narayana polynomials. In the original version of  \cite{Reiner2005Charney}, Reiner and Welker asked whether $N_{D_{n}}(x)$ have only real zeros. Later, Br\"{a}nd\'{e}n  \cite{Braenden2006linear} gave an affirmative answer to this question. In fact, Br{\"a}nd{\'e}n \cite{Braenden2006linear} obtained a more general result as follows.

\begin{thm}[{\cite[Theorem 7.1]{Braenden2006linear}}]\label{branthm}
Given a positive integer $n\geq 2$, let $\alpha, \beta \in \R$ be such that $\alpha \geq 0, \alpha +\beta \geq  0$\footnote{Br{\"a}nd{\'e}n \cite{Braenden2006linear} used the condition $2\alpha +\beta > 0$, while, for $n=4, \alpha =1, \beta=-19/10$, it is easy to verify that $F_n^{(\alpha,\beta)}(x)$ has a pair of complex zeros.}. Then the polynomial
\begin{align}\label{eq-ref}
F_n^{(\alpha,\beta)}(x):=\alpha  N_{B_{n}}(x) + \beta nx  N_{A_{n-2}}(x),
\end{align}
has only real zeros.
\end{thm}

Taking $\alpha=1$ and $\beta=-1$ in \eqref{eq-ref}, we get $N_{D_{n}}(x)$.
Besides, one can check that the $h$-polynomial
has only real zeros  for each of the exceptional cases.
As a corollary of Theorem \ref{branthm}, Br{\"a}nd{\'e}n \cite{Braenden2006linear} proved that the $h$-polynomial $N_W(x)$ has only real zeros
for any finite Weyl group $W$.

The polynomials $N_{n,m}(x)$ arose in the study of the infinite log-concavity of the Boros-Moll polynomials.
The Boros-Moll polynomials was introduced by Boros and Moll \cite{Bormol2004} in their study of a quartic integral, and they obtained the following expression for the Boros-Moll polynomials:
\begin{align*}
P_n(x)=2^{-2n}\sum_j 2^j{2n-2j\choose n-j}{n+j\choose j}(x+1)^j.
\end{align*}
Recall that a finite nonnegative sequence $\{a_k\}_{k=0}^n$ is said to be log-concave if
$$a_k^2-a_{k+1}a_{k-1}\geq 0,\qquad \mbox{for $0\leq k\leq n$},$$
where, for convenience, we set $a_{-1}=0$ and $a_{n+1}=0$.
We say that it is infinitely log-concave if for any $i\geq 1$ the $i$-th iterative sequence $\{\mathcal{L}^i(a_k)\}_{k=0}^n$
is nonnegative, where $\mathcal{L}$ is the operator acting on $\{a_k\}_{k=0}^n$ as follows
$$\mathcal{L}(a_k)=a_k^2-a_{k+1}a_{k-1}, \qquad \mbox{for $0\leq k\leq n$}.$$
We say that a polynomial
$$f(x)=\sum_{k=0}^n a_kx^k$$
is infinitely log-concave if its coefficient sequence $\{a_k\}_{k=0}^n$ is infinitely log-concave.
Boros and Moll proposed the following conjecture.
\begin{conj}[\cite{Bormol2004}]\label{boros-moll-conj}
The polynomial $P_n(x)$ is infinitely log-concave.
\end{conj}

The log-concavity of $P_n(x)$  was first conjectured by Moll \cite{moll2002},
and then was proved
by Kauers and Paule \cite{KauPau2007}. The $2$-fold log-concavity of $P_n(x)$
was proved by Chen and Xia \cite{Chx2010}.
Br\"{a}nd\'{e}n \cite{Bran2011} proposed an innovative approach
to the higher-fold log-concavity of $P_n(x)$. He conjectured the real-rootedness
of some variations
of $P_n(x)$,
from which its $3$-fold log-concavity can be deduced. Br\"{a}nd\'{e}n's
conjectures were later
confirmed by Chen, Dou and Yang \cite{chendouyang}. While Conjecture
\ref{boros-moll-conj} is open,
Br\"{a}nd\'{e}n \cite{Bran2011} has proved the infinite log-concavity of real-rooted polynomials,
which was independently conjectured by Stanley, McNamara and Sagan \cite{mcnsag2010}, and Fisk \cite{fisk2008}.

\begin{thm}[{\cite{Bran2011}}] \label{bran-thm}
If
$$f(x)=\sum_{k=0}^n a_k x^k$$ is a real-rooted polynomial with nonnegative coefficients, then so does the polynomial
$$\sum_{k=0}^n (a_k^2-a_{k-1}a_{k+1})x^k.$$
\end{thm}

The well known Newton's inequality states that if a polynomial $f(x)$ has only
real zeros, then it must
be log-concave. Therefore, Theorem \ref{bran-thm} implies the infinite
log-concavity of the real-rooted polynomials. Motivated by Br\"{a}nd\'{e}n's
theorem, we are led to study the real-rootedness of the following polynomial:
$$Q_n(x)=\sum_{k=0}^n (d_k(n)^2-d_{k-1}(n)d_{k+1}(n)) x^k,$$
where
$$d_k(n)=2^{-2n}\sum_{j=k}^n 2^j{2n-2j\choose n-j}{n+j\choose j}{j\choose k}$$
is the coefficient of $x^k$ in the Boros-Moll polynomial $P_n(x)$. We have the following conjecture.
\begin{conj}\label{conj-yang}
For any $n\geq 1$, the polynomial $Q_n(x)$ has only real zeros.
\end{conj}
Since the log-concavity of $P_n(x)$ is known, by Theorem \ref{bran-thm}, Conjecture \ref{conj-yang}
would imply Conjecture \ref{boros-moll-conj}. Note that
the polynomial $Q_n(x)$ may be rewritten as
\begin{align*}
Q_n(x)&=\sum_{i=0}^n\sum_{j=0}^n 2^{i+j}{2n-2i\choose n-i}{2n-2j\choose n-j}{n+i\choose i}{n+j\choose j} N_{i,j}(x),
\end{align*}
where $N_{i,j}(x)$ is the Narayana polynomial defined by \eqref{eq-nara4}.
The numerical evidence suggests that the polynomial $N_{n,m}(x)$ has only real zeros for any $n$ and $m$.

The remainder of this paper is organized as follows. In the next section, we
shall give an
overview of some tools which will be used to prove the
real-rootedness of $F_n^{(\alpha,\beta)}(x)$ and $N_{n,m}(x)$. In Section \ref{section-3}, we
shall give a new proof of Theorem \ref{branthm}. In Section \ref{section-4}, we
shall prove that the polynomial $N_{n,m}(x)$ is a real-rooted polynomial.

\section{Preliminary results}

The results contained in this section serve as a reference point used in Sections \ref{section-3}  and \ref{section-4}.

Let us first introduce the definition of interlacing. Given two real-rooted polynomials $f(x)$ and $g(x)$ with positive leading coefficients,  We say that {$g(x)$ interlaces $f(x)$}, denoted $g(x)\sep f(x)$, if
\begin{align*}
\cdots\le s_2\le r_2\le s_1\le r_1,
\end{align*}
where $\{r_i\}$ and $\{s_j\}$ are the sets of zeros
of $f(x)$ and $g(x)$, respectively.
We say that $g(x)$ strictly interlaces $f(x)$, denoted $g(x)\prec f(x)$, if, in addition, all the inequalities concerned are strict.

Liu and Wang \cite{Liu2007unified} obtained the following sufficient condition to determine whether two polynomials have interlaced zeros.

\begin{thm}[{\cite[Theorem 2.3]{Liu2007unified}}]\label{liuwangthm2}
Let $F(x),f(x),g_1(x),\ldots,g_k(x)$ be polynomials with real coefficients satisfying the following conditions.
\begin{itemize}
\item[(a)] There exist some polynomials
$\phi(x),\varphi_1(x),\ldots,\varphi_k(x)$  with real coefficients such that
\begin{align}\label{eq-recuu}
F(x)=\phi(x)f(x)+\varphi_1(x)g_1(x)+\cdots+\varphi_k(x)g_k(x)
\end{align}
and $\deg F(x)=\deg f(x)$ or $\deg F(x)=\deg f(x)+1$;

\item[(b)] The polynomials $f(x),g_1(x),\ldots,g_k(x)$ are real-rooted polynomials, and moreover $g_j(x)\sep f(x)$ for each $1\leq j\leq k$;

\item[(c)] The leading coefficients of $F(x)$ and $g_j(x)$ have the same sign for each $1\leq j\leq k$.
\end{itemize}
Suppose that $\varphi_j(r)\leq 0$ for each $j$ and each zero $r$ of $f(x)$. Then $F(x)$ has only real zeros and $f(x)\sep F(x)$.
\end{thm}

We shall use the above result to prove the real-rootedness of $F_n^{(\alpha,\beta)}(x)$ and $N_{n,m}(x)$.
The key point is to prove certain recurrence relations related to these
polynomials. As will be shown later, the coefficients of these
recurrence relations look complicated. Koutschan (private communication)
pointed out that these recurrence relations can be easily verified by using the
Mathematica package \textit{HolonomicFunctions},  see
\cite{Koutsch2009Advanced}. To be self-contained, we give an example to illustrate the use of
this package.
It is well known that the binomial coefficients satisfy the following recurrence relation:
\begin{align*}
 \binom{n}{k}=\binom{n-1}{k}+\binom{n-1}{k-1}.
\end{align*}
This can be proved in the following way by using the package:
\begin{enumerate}

\item[1.]  Convert the recurrence $\binom{n}{k}=\binom{n-1}{k}+\binom{n-1}{k-1}$ to an Ore
 polynomial in the Ore algebra:
 \begin{itemize}
 \item[]
    \begin{mma}
	 \In |rec| = 
|ToOrePolynomial|[|f|[n,k]-|f|[n-1,k]-|f|[n-1,k-1],|f|[n,k]]\\
	 \Out S_n S_k-S_k-1\\
    \end{mma}
\end{itemize}

\item[2.] Generate a (Gr\"oebner) basis of the set $A$
of all recurrence/differential relations that the input satisfies using the command Annihilator:
 \begin{itemize}
 \item[]
    \begin{mma}
	 \In  |ann|=|Annihilator|[|Binomial|[n,k],{|S|[n],|S|[k]}]\\
	\Out  \{(1+k)S_k+(k-n),(1-k+n)S_n+(-1-n)\}\\
    \end{mma}
\end{itemize}
\item[3.] Reduce the Ore polynomial $rec$ modulo $A$ using the command 
OreReduce.
If it returns $0$, then $rec$ is an element of the set $A$ and hence
 the recurrence relation is valid.
 \begin{itemize}
 \item[]
    \begin{mma}
	 \In  |OreReduce|[|rec|,|ann|]\\
	\Out 0\\
    \end{mma}
\end{itemize}
\end{enumerate}

\section{The polynomials $F_n^{(\alpha,\beta)}(x)$}\label{section-3}

The aim of this section is to give a new proof of Theorem \ref{branthm}. In 
fact, we shall show that $$F_n^{(\alpha,\beta)}(x) \sep 
F_{n+1}^{(\alpha,\beta)}(x),\quad \mbox{for any $n\geq 2$.}$$
We first derive recurrence relations for $F_n^{(\alpha,\beta)}(x)$.
For the convenience, let
\begin{align*}
T_1(\alpha,\beta,n)&=(4n-2)\alpha+n\beta,\\
T_2(\alpha,\beta,n)&=(4n-1)\alpha+n\beta,\\
T_3(\alpha,\beta,n)&=4n\alpha+(n+1)\beta,\\
T_4(\alpha,\beta,n)&=6n\alpha+(n+1)\beta,\\
T_5(\alpha,\beta,n)&=(4n+1)\alpha +(n+1)\beta.
\end{align*}

\begin{thm}\label{lem:rec}
For $n\geq 2$, we have the following recurrence relation:
\begin{align}
F_{n+1}^{(\alpha,\beta)}(x)&=\left(\frac{2T_2(\alpha,\beta,n)}{T_1(\alpha,\beta,
n)}x+\frac
{2(n-1)T_5(\alpha,\beta,n)}{(n+1)T_1(\alpha,\beta,n)}\right)F_n^{(\alpha,\beta)}
(x)
\notag  \\[6pt]
&\quad -\,
\frac{(n-2)T_3(\alpha,\beta,n)}{(n+1)T_1(\alpha,\beta,n)}(x-1)^2F_{n-1}^{(\alpha
,\beta)}(x)  \notag   \\[6pt]
& \quad -\,  \frac{2T_4(\alpha,\beta,n)}{n(n+1)T_1(\alpha,\beta,n)}x(x-1)
{(F_n^{(\alpha,\beta)}(x))}'.\label{F-rec}
\end{align}
\end{thm}

\begin{proof}
Note that $F_n^{(\alpha,\beta)}(x)$ has an explicit formula as follows:
\begin{align}\label{formula}
F_n(\alpha,\beta)=\sum_{k=0}^n
\left(\alpha\binom{n}{k}^2+\beta\frac{n}{n-1}\binom{n-1}{k-1}\binom{n-1}{k}
\right)x^k.
\end{align}
We shall prove the following equivalent form of \eqref{F-rec}:
\begin{align*}
n(n+1)T_1F_{n+1}^{(\alpha,\beta)}(x)&=\left({2n(n+1)T_2}x+
{2n(n-1)T_5}\right)F_n^{(\alpha,\beta)}(x)
\notag  \\[6pt]
&\quad -\,
{n(n-2)T_3}(x-1)^2F_{n-1}^{(\alpha,\beta)}(x)  \notag   \\[6pt]
& \quad -\,  {2T_4}x(x-1)
{(F_n^{(\alpha,\beta)}(x))}',
\end{align*}
where $T_i$ represents $T_i(\alpha,\beta,n)$ for each $1\leq i\leq 5$.
First, we convert this recurrence to an Ore
 polynomial $rec$ and compute a (Gr\"oebner) basis $ann$ of the set $A$
of all recurrence/differential relations that $F_n^{(\alpha,\beta)}(x)$ 
satisfies:
 \begin{itemize}
 \item[]
    \begin{mma}
    \In  |Clear|[|f|];\\
    \In |rec|=|ToOrePolynomial|[(2*n*(n+1)*T2*x+2*n*(n-1)*T5)*|f|[n,x]
             -n*(n-2)*T3*(x-1)^2*|f|[n-1,x]-2*T4*x*(x-1)*|D|[|f|[n,x],x]-
             n*(n+1)*T1*|f|[n+1,x]/.|MapThread|[|Rule|,\{\{T1,T2,T3,T4,T5\},
             \{(4*n-2)*\alpha+n*\beta,(4*n-1)*\alpha+n*\beta,
             4*n*\alpha+(n+1)*\beta,6*n*\alpha+(n+1)*\beta,
             (4*n+1)*\alpha+(n+1)*\beta\}\}],|f|[n,x]];\\
    \In |ann|=|Annihilator|[|Sum|[(\alpha*|Binomial|[n,k]^2+\beta*
            n/(n-1)*|Binomial|[n-1,k-1]*|Binomial|[n-1,k])*x^k,\{k,0,n\}],
            {|S|[n],|Der|[x]}];\\
   \end{mma}
\end{itemize}

Then reduce the Ore polynomial $rec$ modulo $A$.
 \begin{itemize}
 \item[]
    \begin{mma}
     \In  |OreReduce|[|rec|,|ann|]\\
    \Out  0\\
    \end{mma}
\end{itemize}
The output is $0$, as desired. This completes the proof.
\end{proof}

We also need the following lemma to prove Theorem \ref{branthm}.

\begin{lem}\label{lemm-nopos} Suppose that $\alpha>0$ and $\alpha+\beta\geq 0$. 
Then, for any $n\geq 3$,
all the coefficients of $F_n^{(\alpha,\beta)}(x)$ are nonnegative.
\end{lem}

\begin{proof} Note that, for $0\leq k\leq n$,
\begin{align*}
[x^k]F_n^{(\alpha,\beta)}(x)=&\alpha\binom{n}{k}^2+\beta\frac{n}{n-1}\binom{n-1}
{k-1}\binom{n-1}{k}\\[5pt]
=&\alpha\left(\binom{n}{k}^2-\frac{n}{n-1}\binom{n-1}{k-1}\binom{n-1}{k}
\right)\\[5pt]
&\quad + (\alpha+\beta)\frac{n}{n-1}\binom{n-1}{k-1}\binom{n-1}{k}\\[5pt]
\geq 
&\alpha\left(\binom{n}{k}^2-\frac{n}{n-1}\binom{n-1}{k-1}\binom{n-1}{k}\right)\\
[5pt]
= &\alpha\binom{n}{k}^2\left(1-\frac{k(n-k)}{n(n-1)}\right)\geq 0.
\end{align*}
This completes the proof.
\end{proof}

We proceed to prove the real-rootedness of $F_n^{(\alpha,\beta)}(x)$.

\noindent \textit{Proof of Theorem \ref{branthm}.} We use induction on $n$.
We may assume that $\alpha>0$. The hypothesis that $\alpha+\beta\geq 0$ implies 
that
both $F_2^{(\alpha,\beta)}(x)$ and $F_3^{(\alpha,\beta)}(x)$ are real-rooted. 
Moreover,
we have
\begin{align*}
F_2^{(\alpha,\beta)}(x)&=\alpha\left(x+\frac{2\alpha+\beta-\sqrt{
(\alpha+\beta)(3\alpha+\beta)}}{\alpha}\right)
\left(x+\frac{2\alpha+\beta+\sqrt{(\alpha+\beta)(3\alpha+\beta)}}{\alpha}\right)
,\\[10pt]
F_3^{(\alpha,\beta)}(x)&=\alpha(x+1)\left(x+\frac{8\alpha+3\beta-\sqrt{
3(2\alpha+\beta)(10\alpha+3\beta)}}{2\alpha}\right)\\[10pt]
&\qquad 
\times\left(x+\frac{8\alpha+3\beta-\sqrt{3(2\alpha+\beta)(10\alpha+3\beta)}}{
2\alpha}\right).
\end{align*}
It is routine to verify that $F_2^{(\alpha,\beta)}(x)\sep 
F_3^{(\alpha,\beta)}(x)$. The details are tedious and will not be given here.

Assume that $F_{n-1}^{(\alpha,\beta)}(x)$ and $F_{n}^{(\alpha,\beta)}(x)$ have 
only real zeros, and $F_{n-1}^{(\alpha,\beta)}(x)\sep 
F_{n}^{(\alpha,\beta)}(x)$. We proceed to verify that
$F_{n+1}^{(\alpha,\beta)}(x)$  has only real zeros and
$F_{n}^{(\alpha,\beta)}(x)\sep F_{n+1}^{(\alpha,\beta)}(x)$. We see that the 
recurrence relation \eqref{F-rec}
is of the form \eqref{eq-recuu} in Theorem \ref{liuwangthm2} with $k=2$, where
\begin{align*}
F(x)&=F_{n+1}^{(\alpha,\beta)}(x),\\[5pt]
f(x)&=F_{n}^{(\alpha,\beta)}(x),\\[5pt]
g_1(x)&=F_{n-1}^{(\alpha,\beta)}(x),\\[5pt]
g_2(x)&=(F_{n}^{(\alpha,\beta)}(x))',\\[5pt]
\phi(x)&=\left(\frac{2T_2(\alpha,\beta,n)}{T_1(\alpha,\beta,n)}x+\frac
{2(n-1)T_5(\alpha,\beta,n)}{(n+1)T_1(\alpha,\beta,n)}\right),\\[5pt]
\varphi_1(x)&=-\,
\frac{(n-2)T_3(\alpha,\beta,n)}{(n+1)T_1(\alpha,\beta,n)}(x-1)^2,\\[5pt]
\varphi_2(x)&=-\,  \frac{2T_4(\alpha,\beta,n)}{n(n+1)T_1(\alpha,\beta,n)}x(x-1).
\end{align*}
Recall that we are always assuming that $\alpha>0$ and  $\alpha+\beta\geq 0$. 
With this assumption, all of 
$T_1(\alpha,\beta,n),T_3(\alpha,\beta,n),T_4(\alpha,\beta,n)$ are positive  for 
$n\geq 3$. By Lemma \ref{lemm-nopos}, all the real zeros of 
$F_{n}^{(\alpha,\beta)}(x)$ are non-positive. Therefore,
for any zero $r$ of $F_{n}^{(\alpha,\beta)}(x)$, we have
\begin{align*}
\varphi_1(r)\leq 0, \qquad \varphi_2(r)\leq 0.
\end{align*}
By Theorem \ref{liuwangthm2}, the polynomial $F_{n+1}^{(\alpha,\beta)}(x)$  has 
only real zeros, and moreover
$F_{n}^{(\alpha,\beta)}(x)\sep F_{n+1}^{(\alpha,\beta)}(x)$.  This completes 
the proof.
\qed

\section{The polynomials $N_{n,m}(x)$}\label{section-4}

In this section, we shall prove the real-rootedness of the polynomial
$N_{n,m}(x)$. Similar to the proof of the real-rootedness of 
$F_{n}^{(\alpha,\beta)}(x)$,
we first derive certain recurrence relations for these polynomials.
For nonnegative integers $t$ and $n$, let
\begin{align}\label{eq-var}
\overline{N}^{(t)}_n(x)=N_{n+t,n}(x), \quad
\underline{N}^{(t)}_n(x)=N_{n,n+t}(x).
\end{align}

We have the following recurrence relation for $\overline{N}^{(t)}_n(x)$.

\begin{thm}\label{lem:Anm-0}
For nonnegative integers $t$ and $n\geq 1$, we have
\begin{align}\label{Jn-rec-0}
\overline{N}^{(t)}_{n+1}(x)&=\frac{a_0
+a_1x
+a_2x^2}{(n+t+1)(n+3)(c_0+c_1x)}\overline{N}^{(t)}_{n}(x)
-\frac{n(n+t)(x-1)^2(b_0+b_1x)}{
(n+t+1)(n+3)(c_0+c_1x)}\overline{N}^{(t)}_{n-1}(x),
\end{align}
where
\begin{align*}
a_0&=-(2n+3)(n+t)(n+t+1),\\
a_1&=3t(t-2)(t+1)^2/2+t(t-2)(t^2+7t+5)n\\
   &\quad +3t(t-2)(t+2)n^2+2t(t-2)n^3\\
a_2&=t^2(t-1)(t+1)^2/2+(t-1)(2t^3+3t^2+t-3)n\\
& \quad +(t-1)(3t^2+3t-5)n^2+2(t-1)^2n^3,\\
b_0&=-n-1-t,\\
b_1&=(t-1)^2n+(t-1)t^2/2+(t-1)^2,\\
c_0&=-n-t,\\
c_1&=(t-1)^2n+(t-1)t^2/2.
\end{align*}
\end{thm}

\begin{proof}
We shall prove an equivalent form of this recurrence relation, which is 
obtained by
multiplying $(n+t+1)(n+3)(c_0+c_1x)$ on both sides of \eqref{Jn-rec-0}. This 
could be converted into an Ore
 polynomial as follows:
 \begin{itemize}
 \item[]
    \begin{mma}
	\In |rec|=|ToOrePolynomial|[(a0+a1*x+a2*x^2)*|f|[n]-(n*(n+t)*
	        (x-1)^2*(b0+b1*x))*|f|[n-1]-(n+3)*(n+t+1)*(c0+c1*x)*|f|[n+1]
	        /.|MapThread|[|Rule|,\{\{a0,a1,a2,b0,b1,c0,c1\},\{-(2*n+3)*(n+t)
	        *(n+t+1),3*t*(t-2)*(t+1)^2/2+t*(t-2)*(t^2+7t+5)*n+3*t*(t-2)
	        *(t+2)*n^2+2*t*(t-2)*n^3,t^2*(t-1)*(t+1)^2/2+(t-1)*(2*t^3+
	        3*t^2+t-3)*n+(t-1)*(3*t^2+3*t-5)*n^2+2*(t-1)^2*n^3,-n-1-t,
	        (t-1)^2*n+(t-1)*t^2/2+(t-1)^2,-n-t,(t-1)^2*n+(t-1)*t^2/2\}\}],
	        |f|[n]];\\
    \end{mma}
\end{itemize}
Then compute a (Gr\"oebner) basis $ann$ of the set
of all recurrence/differential relations that $\overline{N}^{(t)}_{n}(x)$ 
satisfies, and
reduce the Ore polynomial $rec$ modulo $ann$:
 \begin{itemize}
 \item[]
    \begin{mma}
	  \In |ann|=|Annihilator|[|Sum|[(|Binomial|[n+t,k]*|Binomial|[n,k]-
      |Binomial|[n+t,k+1]*|Binomial|[n,k-1])*x^k,\{k,0,n+t\}],|S|[n]];\\
	 \In |OreReduce|[|rec|,|ann|]\\
	\Out 0\\
    \end{mma}
\end{itemize}
The output is $0$, as desired.
This completes the proof.
\end{proof}

We now prove the real-rootedness of $\overline{N}^{(t)}_n(x)$.

\begin{thm}\label{overlineNP}
For any $n,t\geq 0$, the polynomial $\overline{N}^{(t)}_n(x)$
has only real zeros. If $t\geq 2$, then $\overline{N}^{(t)}_n(x)$
has one and only one positive zero.
\end{thm}

\begin{proof}
Note that both the polynomials $\overline{N}^{(0)}_n(x)$ and
$\overline{N}^{(1)}_n(x)$ are the classical Narayana polynomial, which is known
to be real-rooted.

We proceed to consider the case of $t\geq2$.
Assume that $\overline{N}^{(t)}_{n-1}(x) \sep \overline{N}^{(t)}_{n}(x)$.
We see that the recurrence relation \eqref{Jn-rec-0}
is of the form \eqref{eq-recuu} in Theorem \ref{liuwangthm2} with $k=1$, where
\begin{align*}
F(x)&=\overline{N}^{(t)}_{n+1}(x),\\
f(x)&=\overline{N}^{(t)}_{n}(x),\\
g_1(x)&=\overline{N}^{(t)}_{n-1}(x),\\
\phi(x)&=\frac{a_0
+a_1x
+a_2x^2}{(n+t+1)(n+3)(c_0+c_1x)},\\
\varphi_1(x)&=
-\frac{n(n+t)(x-1)^2(b_0+b_1x)}{
(n+t+1)(n+3)(c_0+c_1x)}.
\end{align*}

For any $n\geq 0$ and $t\geq 2$, $\overline{N}^{(t)}_n(x)$ is polynomial in $x$
of degree $n+1$, and for any
$0\leq k\leq n+1$, the coefficient of $x^k$ in
$\overline{N}^{(t)}_n(x)$ is
\begin{align*}
{\binom{n+t}{k}\binom{n}{k}}-{\binom{n+t}{k+1}\binom{n}{k-1}}&=\frac{n+1-kt}{
(n+1)(k+1)}{
\binom {n+t}{k}\binom{n+1}{k}}.
\end{align*}
Therefore, the number of changes in sign of the coefficients is $1$.
By Descartes' rule, the polynomial $\overline{N}^{(t)}_n(x)$ has at most one
positive zeros.
Moreover, we see that
\begin{align*}
[x^0]\overline{N}^{(t)}_n(x)=1>0, \quad [x^{n+1}]\overline{N}^{(t)}_n(x)=
-\binom{n+t}{n+2}<0.
\end{align*}
Thus, the polynomial $\overline{N}^{(t)}_n(x)$ has one and only one positive
zero.

We claim that $\overline{N}^{(t)}_n(x)$ has $n$ negative zeros, and moreover,
for any $n\geq 1$,
\begin{align*}
r_{n+1}^{(n+1)}<r_n^{(n)}<r_{n}^{(n+1)}<r_{n-1}^{(n)}<\cdots
<r_2^{(n)}<r_2^{(n+1)}<r_1^{(n)}<r_1^{(n+1)}<0,
\end{align*}
where $\{r_k^{(n)}\}_{k=0}^n$ and $\{r_k^{(n+1)}\}_{k=0}^{n+1}$ are the negative
zeros of $\overline{N}^{(t)}_n(x)$ and $\overline{N}^{(t)}_{n+1}(x)$
respectively.

The proof of the claim is by induction on $n$. To this end, we need to determine
the sign of the coefficient of $\overline{N}^{(t)}_{n-1}(x)$ in the recurrence
\eqref{Jn-rec-0}. For any $x<0, n\geq 1$ and  $t\geq 2$, it is easy to show that
\begin{align*}
-\frac{n(n+t)(x-1)^2(b_0+b_1x)}{
(n+t+1)(n+3)(c_0+c_1x)}<0.
\end{align*}
By \eqref{Jn-rec-0}, we see that $\overline{N}^{(t)}_{2}(r_1^{(1)})<0$.
Moreover, we have $\overline{N}^{(t)}_{2}(0)=1>0$ and
$\overline{N}^{(t)}_{2}(-\infty)>0$. Thus,
$r_2^{(2)}<r_1^{(1)}<r_1^{(2)}<0$, as claimed. Assume the claim is true for $n$.
From \eqref{Jn-rec-0} we deduce that
$$(-1)^k\overline{N}^{(t)}_{n+1}(r_k^{(n)})>0, \quad \mbox{for any $1\leq k\leq
n$}.$$
Moreover, we have $\overline{N}^{(t)}_{n+1}(0)=1>0$ and
$(-1)^{n+1}\overline{N}^{(t)}_{n+1}(-\infty)>0$.
Thus, the polynomials $\overline{N}^{(t)}_{n+1}(x)$ has $n+1$ negative zeros
$\{r_k^{(n+1)}\}_{k=0}^{n+1}$, and
moreover, for each $1\leq k\leq n$, we have
$r_{k+1}^{(n+1)}<r_k^{(n)}<r_{k}^{(n+1)}$, as claimed.
This completes the proof.
\end{proof}

The polynomials $\underline{N}^{(t)}_n(x)$ satisfy the following recurrence 
relation.

\begin{thm}\label{lem:Anm}
For nonnegative integers $t$ and $n\geq 1$, we have
\begin{align}\label{Jn-rec}
\underline{N}^{(t)}_{n+1}(x)&=\frac{a_0
+a_1x
+a_2x^2}{(n+t+3)(n+1)(c_0+c_1x)}\underline{N}^{(t)}_{n}(x)-\frac{
n(n+t)(x-1)^2(b_0+b_1x)}{
(n+t+3)(n+1)(c_0+c_1x)}\underline{N}^{(t)}_{n-1}(x),
\end{align}
where
\begin{align*}
a_0&=-(2n^3+(2t+5)n^2+(2t+3)n),\\
a_1&=(2t(t+2)n^3+3t(t+2)^2n^2+(t(t+2)(t^2+5t+5))n\\
& \quad +(t(t+1)(t+2)(t+3)/2)),\\
a_2&=(t+1)((2t+2)n^3+(3t^2+9t+5)n^2+(2t+3)(t^2+3t+1)n\\
& \quad +t(t+1)(t+2)(t+3)/2),\\
b_0&=-(n+1),\\
b_1&=(t+1)^2n+(t+1)(t^2+4t+2)/2,\\
c_0&=-n,\\
c_1&=(t+1)^2n+t(t+1)(t+2)/2.
\end{align*}
\end{thm}

\begin{proof}
The proof is similar to that of Theorem \ref{lem:Anm-0}. First, we convert 
\eqref{Jn-rec} to an Ore
 polynomial:
 \begin{itemize}
 \item[]
    \begin{mma}
	\In  |rec|=|ToOrePolynomial|[(a0+a1*x+a2*x^2)*|f|[n]-(n*(n+t)
	          *(x-1)^2*(b0+b1*x))*|f|[n-1]-(n+t+3)*(n+1)*(c0+c1*x)
	          *|f|[n+1]/.|MapThread|[|Rule|,\{\{a0,a1,a2,b0,b1,c0,c1\},
	          \{-(2*n^3+(2*t+5)*n^2+(2*t+3)*n),2*t*(t+2)*n^3+3*
	          t*(t+2)^2*n^2+(t*(t+2)*(t^2+5*t+5))*n+t*(t+1)*
	          (t+2)*((t+3)/2),(t+1)*((2*t+2)*n^3+(3*t^2+9*t+5)
	          *n^2+(2*t+3)*(t^2+3*t+1)*n+t*(t+1)*(t+2)*((t+3)/2))
	          ,-(n+1),(t+1)^2*n+(t+1)*((t^2+4*t+2)/2),-n,(t+1)^2
	          *n+t*(t+1)*((t+2)/2)\}\}],|f|[n]];\\
    \end{mma}
\end{itemize}
Then compute a (Gr\"oebner) basis $ann$ of the set
of all recurrence/differential relations that $\underline{N}^{(t)}_n(x)$ 
satisfies, and
reduce the Ore polynomial $rec$ modulo $ann$:
 \begin{itemize}
 \item[]
    \begin{mma}
	 \In  |ann|=|Annihilator|[|Sum|[(|Binomial|[n,k]*|Binomial|[n+t,k]
               -|Binomial|[n,k+1]*|Binomial|[n+t,k-1])*x^k,\{k,0,n\}],|S|[n]];\\
	 \In |OreReduce|[|rec|,|ann|]\\
	\Out  0\\
    \end{mma}
\end{itemize}
We have the desired output. This completes the proof.
\end{proof}

Next we come to proving the real-rootedness of $\underline{N}^{(t)}_n(x)$.

\begin{thm}\label{underlineNP} For any $t\geq 0$ and $n\geq 0$, the polynomial $\underline{N}^{(t)}_n(x)$ has only real zeros, and moreover, we have $\underline{N}^{(t)}_n(x)\sep \underline{N}^{(t)}_{n+1}(x)$.
\end{thm}

\begin{proof}
We use induction on $n$.
It is straightforward to verify that
\begin{align*}
\underline{N}^{(1)}_0(x)= 1, \quad \underline{N}^{(1)}_1(x)=1+(t+1)x,\quad
 \underline{N}^{(1)}_0(x) \sep \underline{N}^{(1)}_1(x).
\end{align*}
Assume that $\underline{N}^{(t)}_{n-1}(x) \sep \underline{N}^{(t)}_{n}(x)$.
We see that the recurrence relation \eqref{Jn-rec}
is also of the form \eqref{eq-recuu} in Theorem \ref{liuwangthm2} with $k=1$, where
\begin{align*}
F(x)&=\underline{N}^{(t)}_{n+1}(x),\\
f(x)&=\underline{N}^{(t)}_{n}(x),\\
g_1(x)&=\underline{N}^{(t)}_{n-1}(x),\\
\phi(x)&=\frac{a_0
+a_1x
+a_2x^2}{(n+t+3)(n+1)(c_0+c_1x)},\\
\varphi_1(x)&=-\frac{n(n+t)(x-1)^2(b_0+b_1x)}{
(n+t+3)(n+1)(c_0+c_1x)}\underline{N}^{(t)}_{n-1}(x).
\end{align*}
Here, $a_0,a_1,a_2,b_0,b_1,c_0,c_1$ are given by \eqref{Jn-rec}.
Note that for any $n,t\geq 0$ the coefficients of $\underline{N}^{(t)}_{n+1}(x)$ are nonnegative,
since, for any $0\leq k\leq n$, the coefficient of $x^k$ in $\underline{N}^{(t)}_{n+1}(x)$ is
\begin{align*}
[x^k]\underline{N}^{(t)}_{n+1}(x) &= \binom{n}{k}\binom{n+t}{k}-\binom{n}{k+1}\binom{n+t}{k-1}\\
 &= \left(1-\frac{k}{k+1}\cdot \frac{n-k}{n-k+t+1}\right)\binom{n}{k}\binom{n+t}{k}>0.
\end{align*}

It is clear that for any $x<0$, we have $\varphi_1(x)\leq 0$. By Theorem \ref{liuwangthm2},
the polynomial $\underline{N}^{(t)}_{n+1}(x)$ is real-rooted, and moreover  $\underline{N}^{(t)}_{n}(x) \sep \underline{N}^{(t)}_{n+1}(x)$.
\end{proof}

Combining Theorems \ref{overlineNP} and \ref{underlineNP}, we obtain the following result.

\begin{cor}
For any $m,n\geq 0$, the polynomial $N_{m,n}(x)$ has only real zeros.
\end{cor}

\vskip 3mm
\noindent{\bf Acknowledgements.} This work was supported by the 973 Project, the PCSIRT Project of the Ministry of Education and the National Science Foundation of China.
We would like to thank Dr. Christoph Koutschan for helpful suggestions on the use of the Mathematica package
\textit{HolonomicFunctions}.

\end{document}